\documentclass[12pt]{amsart}

\usepackage{amsmath}
\usepackage{amssymb}

\usepackage{enumitem}
\usepackage{graphicx}

\makeatletter
\@namedef{subjclassname@2020}{%
  \textup{2020} Mathematics Subject Classification}
\makeatother

\usepackage[T1]{fontenc}

\newtheorem{theorem}{Theorem}[section]
\newtheorem{corollary}[theorem]{Corollary}
\newtheorem{lemma}[theorem]{Lemma}

\newtheorem{conjecture}[theorem]{Conjecture}

\theoremstyle{definition}
\newtheorem{definition}[theorem]{Definition}

\numberwithin{equation}{section}

\begin{document}

\baselineskip=17pt

\title[About the Primality of Primorials]{About the Primality of Primorials}

\author[G. Lillie]{George M. Lillie}
\address{Department of Mathematics\\ University of Michigan
\ 
}
\email{lilliege@umich.edu}

\date{}

\begin{abstract}
A primorial prime is a prime number of the form $p_n\# \pm 1$ where $p_n\#$ denotes the product of all primes less than or equal to $p_{n}$, the $n$-th prime. We show that the probability along the lines of Mertens' Theorem that either $p_n\# -1$ or $p_n\# +1$ is prime is $O(n^{-1})$ and that the probability that both $p_n\# -1$ and $p_n\# +1$ are prime is $O(n^{-2})$, for $n>1$. The latter result provides evidence that there are in total three instances where both $p_n\# -1$ and $p_n\# +1$ are prime. We provide proof that numbers of the from $p_n\# \pm 1$ have the highest probability of being prime. 
\end{abstract}



\maketitle

\section{Introduction}
Let $p$ always be a prime and $p_n$ denote the $n$-th prime, starting with $p_1=2$. We denote by $\pi(x)$, $\pi_2(x)$, and $\vartheta(x)$, respectively, the number of primes $p \leqslant  x$, the number of primes $p\leqslant  x$ for which $p+2$ is also prime, and the logarithm of the product of primes $p \leqslant  x$. Let $k=1,2$ and $c\geqslant  1/2$ both be constants. Consider 
\begin{equation*}
p_n\# \equiv \prod_{i=1}^{n}p_i.
\end{equation*}
In 1987, Dubner \cite{5} defined $p_n\#$ as the primorial function, where term ``primorial'' itself draws from the neologistic analogy: \textit{factors} is to \textit{factorials} as \textit{primes} is to \textit{primorials}. Primorial primes are considered to be prime numbers of the form $p_n\# \pm 1$. In a similar fashion, we specify a primorial twin prime pair as a twin prime pair of the form $(p_n\#-1, p_n\#+1)$. Conjecture \ref{cn1} follows.
\begin{conjecture} \label{cn1}
The total expected number of primorial twin prime pairs is approximately three. 
\end{conjecture}
We provide evidence in Section 4 that $n=2,3,5$ give the three primorial twin prime pairs in the above conjecture. Definition \ref{df1} expands on the notion of primorials below. 
\begin{definition} \label{df1}
Define a universal primorial to be an odd, positive integer of the form
$$
K p_n\#+g, \hspace{0.5cm} \mbox{$K \in \left\{\frac{N}{2}: N=1,2,...\right\}$},
$$
where
\[ g=
\left\{ \begin{array}{ll}
         \mbox{$\pm 1$} & \mbox{for even $N$};\\
         \mbox{$2$ or $4$} & \mbox{for odd $N$}.\end{array} \right. \]
\end{definition}

Since any positive integer can be written as a universal primorial, Conjecture \ref{up} below is equivalent to the Twin Prime Conjecture.
\begin{conjecture} \label{up}
There are expected to be infinitely many numbers $x=K p_n\#+g$ such that both $x$ and $x+2$ are prime.
\end{conjecture}

Now, we consider the probability that a given universal primorial is prime. Formally, the probability that a number is prime is either zero or one, but the Prime Number Theorem (PNT) maintains that $\pi(x)\sim x/ \log x$ as $x\to \infty,$ suggesting that the probability a number roughly the value of $x$ is prime is asymptotic to $(\log x)^{-1}$. This result differs from the result of an intuitive approach to determining the primality of a given number by a constant factor (namely, $e^{-\gamma}$, where $\gamma \approx 0.577$ is the Euler–Mascheroni constant), as shown next. 

The probability that $p$ does not divide a reasonably large $x \in \mathbb{N^+}$ is simply $(p-1)/p$. Thus, the probability that no $p \leqslant  x^c$ divides $x$ is given by Eq. \eqref{mer} below in Mertens' so-called third theorem \cite{10}.
\begin{theorem}[\bfseries Mertens]
For $x \in \mathbb{N^+}$, we have 
\begin{equation} \label{mer}
M^*(x^c)\equiv \prod_{p \leqslant  x^c} \frac{p-1}{p}=\frac{e^{-\gamma}}{c \log x}(1+o(1)).
\end{equation}
\end{theorem}
($M^*(x^c)$ is not to be confused with Mertens' Function $M(n),$ which denotes the sum of the M\"obius function $\mu(z)$ over all integers $1\leqslant  z \leqslant  n$.) We look to make improvements on such intuitive probability calculations for primality by considering the following definition. 
\begin{definition}
For $x \in \mathbb{N^+}$, define the set
\begin{equation*}
\mathcal{S}_{x, b}=\{p : p \mid x-b \mbox{ and } p \nmid b\}, \hspace{0.5cm} b \in \mathbb{N^+}, 1\leqslant  b <x.
\end{equation*}
\end{definition}
The set $\mathcal{S}_{x, b}$ is composed of a certain arrangement of primes that are known to not divide $x$, and, therefore, such primes can be excluded from related probability calculations for primality. Since we are only considering universal primorials, we choose to make the following definition. 
\begin{definition} \label{dpol} Let $x>1$ be a universal primorial and let $$\mathcal{U}=\{K p_n\#+g: \forall n, N \in \mathbb{N^+} \}.$$ Define $L: \mathcal{U} \to \mathbb{R^+}$. We have
\begin{equation*}
L_k(p_n; n, x) \equiv \frac{1}{\prod\limits_{ p \in \mathcal{S}_{x, (b=g)}} \frac{p-k}{p}}\prod_{\substack{p \leqslant  x^c \\ p > k}} \frac{p-k}{p}.
\end{equation*}
\end{definition}
Along the lines of Mertens' Theorem, $L_1(p_n; n) $ denotes the probability that either $p_n\#-1$ or $p_n\#+1$ is prime, and $L_2(p_n; n) $ denotes the probability that both $p_n\#-1$ and $p_n\#+1$ are prime.\footnote{The word ``probability'' is meant to be interpreted as $L_k(\dots)$ throughout the rest of this paper unless specified otherwise.} Theorem \ref{U} follows.

\begin{theorem}\label{U} 
For any primorial $p_n\#$, $n>1$, we have
\begin{equation} \label{pnew} L_k(p_n; n) = O\left(n^{-k} \right).
\end{equation}
In particular, 
\begin{equation} \label{asim}
L_k(p_n; n) \sim \theta_kn^{-k}, \hspace{0.5cm} (15/16)^{k-1}<\theta_k \leqslant  2^{k}.
\end{equation}
\end{theorem}
Lemma \ref{T2} aids in proving Theorem \ref{U} and is stated below (see Lemmata \ref{coo2} and \ref{coo3} in Section 6 for sharper bounds). 
\begin{lemma} \label{T2} For $x \geqslant  599$, 
\begin{equation*}
\left| \frac{1}{\pi(x)}-\frac{\log x}{\vartheta(x)} \right | < 0.30543\frac{\log x}{\vartheta(x)}.
\end{equation*}
\end{lemma}
Theorem \ref{Aa} provides evidence of the asymptotic relation given in Eq. \eqref{asim}.
\begin{theorem} \label{Aa}
For any primorial $p_n\#$, $n>168\ 064$,
\begin{equation}\label{uio1}
\left| L_k(p_n;n)-\theta_k\frac{1}{n^k}\right|<\theta_k\left(\frac{0.0000642}{n}\right)^k.
\end{equation}
\end{theorem}
The following definition is stated for the sake of convenience and in preparation for Theorem \ref{denns}. 
\begin{definition}
For $x \in \mathbb{N^+}$, define $\mathcal{P}_x$ to be the set of all prime numbers less than the square root of $x$: $$\mathcal{P}_x = \{p: p \leqslant  \sqrt{x}\}.$$
\end{definition}
Although numbers $p_n\#\pm1$ become rarely prime, i.e., $L_k(p_n; n) \ll 1$, as $n\to \infty$, Theorem \ref{denns} states that these numbers have the highest probability of being prime and being twin primes compared to all other numbers with the same number of primes less than their square root. This is a reasonable comparison given that only the prime numbers less than the square root of a number determine that number's primality.

\begin{theorem} \label{denns}
Let $x=p_n\#\pm1$, for $n>1$, and let $\overline{x}$ be any integer that satisfies both $\mathcal{P}_{\overline{x}}=\mathcal{P}_{x}$ and $\overline{x} \neq x$. We have $$L_k(p_n; n) >L_k(\overline{x}),\hspace{0.5cm} k=1,2.$$
\end{theorem}

\section{A Lemma on Universal Primorials}

A universal primorial is never divisible by any prime $p \leqslant  p_n$, so one may naturally wonder if a universal primorial $x$ can always be found such that the largest prime less than the square root of $x$ is still less than $p_n$ since such numbers would always be prime. By using factorials as a heuristic, it is clear that there must exist some number $x_0$, where for all universal primorials greater than $x_0$, there does not exist any universal primorial for which this occurs. Lemma \ref{A} shows that $x_0=106$ and is proved next. Note that instead of leaving $N$ to be arbitrary, we fix $N=1$ to minimize $K$ so that $x_0$ to be maximized.
\begin{lemma}\label{A}
Let $x=\frac{1}{2}p_n\#\pm1$. We find 
\begin{equation}p_n\# < \prod_{p \leqslant  \sqrt{x}} p, \hspace{0.5cm} x>106.
\end{equation}
\end{lemma}
\begin{proof}
Suppose that  $|\mathcal{P}_x|=n$. It follows that $p_{n+1}>\sqrt{x}$ where $p_{n+1}$ is the subsequent prime to $p_n$. Substitute $p_n\#/2$ in for $x$ because 
$$1>\sqrt{\frac{1}{2}p_n\#+4} -\sqrt{\frac{1}{2}p_n\#}, \hspace{0.5cm} x>\frac{5}{2}+4.
$$
Thus $p_{n+1}>\sqrt{p_n\#/2}.$ By Bertrand's Postulate \cite{2}, we have $p_{n+1}< 2p_n$, which leads to 
$$y_1^2p_n>\frac{p_n\#}{2p_n}, \hspace{0.5cm} 1< y_1, y_2 < 2.$$
Hence,
$$8>y_1^2y_2>\frac{p_n\#}{2p_np_{n-1}}.$$
So, at most $p_{n-2}=3$, implying that $p_n=7$. Lemma \ref{A} follows. 
\end{proof}

\section{Proof of Main Results}
\subsection{Proof of Lemma \ref{T2} }
\begin{proof}
Let $H: \mathbb{R^+} \to \mathbb{R^+}$ be given by 
$$H(x)= \frac{\vartheta(x)}{\pi(x)\log x}.$$
Consider the sharp result below provided by Trudgian \cite{14}:
$$|\vartheta(x)-x|\underset{x \geqslant  149}{<}x\epsilon_0(x),$$
\begin{equation} \label{hol}
\epsilon_0(x)=\sqrt{\frac{8}{17\pi}}X^{1/2}e^{-X}, \hspace{0.5cm} X=\sqrt{\frac{\log x}{6.455}}.
\end{equation}
We find 
\begin{equation} \label{pd}
 \left| \frac{1}{\pi(x)}-\frac{\log x}{\vartheta(x)} \right | = \frac{\log x}{\vartheta(x)}\left | H(x)-1 \right|.
 \end{equation}
Dusart \cite{6, 7} shows that
\begin{equation}\label{d1}
\frac{x}{\log x} \left(1+\frac{1}{\log x}\right)\underset{x\geqslant  599}{<} \pi(x) \underset{x\geqslant  1}{<} \frac{x}{\log x} \left (1+\frac{1}{\log x}+\frac{2}{\log^2(x)}+\frac{7.59}{(\log x)^3}\right).
\end{equation}
We use values $x \geqslant  599$ in any case to accommodate the second inequality. Let 
\begin{equation*}
\rho(x)=1+\frac{1}{\log x}+\frac{2}{\log^2(x)}+\frac{7.59}{(\log x)^3},
\end{equation*}
and define
\begin{equation*}
\varphi(x) =\left|\frac{1-\epsilon_0(x)}{\rho(x)}-1\right|, \hspace{0.5cm} \mbox{ for $x\geqslant 599.$}
\end{equation*}
It follows straightforwardly that $\varphi(x)$ produces the largest value that $|H(x)-1|$ can obtain for all values greater than or equal to $x$. Lemma \ref{T2} is obtained from $\varphi(599) \approx 0.30543.$
\end{proof}

\subsection{Proof of Theorem \ref{U}}
The statement $\vartheta(x) \sim x$ is equivalent to the PNT.  For sufficiently large $x$, as given by $x \gg 1$, we write $x  \approxeq \vartheta(x)$ to establish that $x$ estimates $\vartheta(x)$ well.  Consequently, 
$$\log(Kp_n\#+g) \approxeq p_n, \hspace{0.5cm} \mbox{$p_n\#  \gg  K, p_n \gg 1.$}$$ Definition \ref{dd1} provides a criteria for classifying $K$.
\begin{definition} \label{dd1}
Let $x$ be universal primorial and define the following function
\begin{equation*} \label{EE}
\alpha(x, k ;p_n)=\bigg(\frac{\log p_n\#}{\log x}\bigg)^k.
\end{equation*}
We say that $K$ is sufficiently small if $\alpha(x, k ;p_n) \sim 1$.
\end{definition}
For large enough $x$ and sufficiently small $K$, $\alpha(x, k ;p_n) \approxeq 1$ is reasonable, which can not necessarily be said if $\alpha(x, k ;p_n) \nsim 1$.

Consider that if $x$ is a universal primorial, then $\mathcal{S}_{x, b}=\{p: p \leqslant  p_n\}$, and we have
\begin{equation}\label{Er}
L_k(p_n; n, x)=\frac{1}{\prod\limits_{ p \leqslant  p_n} \frac{p-k}{p}}\prod_{\substack{p \leqslant  x^c \\ p > k}} \frac{p-k}{p}=\theta_k\left(\frac{\log p_n}{\log x}\right)^k,
\end{equation}
where $\theta_k$ is a constant with partial dependence on $k$. 
For $p_n \gg 1,$ $\log x= \log (Kp_n\#+g)\approx\log (Kp_n\#)$. Moreover,
\begin{equation*}
\bigg(\frac{\log p_n}{\log x}\bigg)^k=\alpha(x, k; p_n)\bigg(\frac{\log p_n}{\vartheta(p_n)}\bigg)^k.
\end{equation*}

Now, it can easily be show that $L_k(p_n; n)=O(n^{-k})$. For instance, from \cite[Theorem 8]{13} (see its Corollary as well), we find weak bounds for $L_k(p_n; n, x)$:
\begin{equation} \label{int5}
\begin{split}
\theta_k\alpha(x, k; p_n) \left(\frac{\log p_n}{\vartheta(p_n)}\right)^k&\frac{1-\frac{1}{\log^2(x)}}{1+\frac{1}{2\log^2(p_n)}}<L_k(p_n; n, x) \\
&<\theta_k\alpha(x, k; p_n) \left(\frac{\log p_n}{\vartheta(p_n)}\right)^k\frac{1+\frac{1}{2\log^2(x)}}{1-\frac{1}{\log^2(p_n)}}, \hspace{0.5cm} \mbox{$n>1$.}
\end{split}
\end{equation}
The far-right term in the upper bound above takes on its largest value of $6.9579$ when $x=3\#-1$. 
Since
\begin{equation} \label{int1}
\frac{1}{[(1+\varphi(x))\pi(x)]^k}<\left(\frac{\log p_n}{\vartheta(p_n)}\right)^k<\frac{1}{[(1-\varphi(x))\pi(x)]^k},
\end{equation}
we have 
$$\left(\frac{0.76603}{n}\right)^k< \left(\frac{\log p_n}{\vartheta(p_n)}\right)^k<\left(\frac{1.4397}{n}\right)^k, \hspace{0.5cm} \mbox{$x\geqslant 599$,}$$
by Lemma \ref{T2}. Now take $K$ to be sufficiently small. Therefore, 
\begin{equation*}
L_k(p_n; n)<6.9579 \theta_k \left(\frac{1.4397}{n}\right)^k=O(n^{-k}).
\end{equation*}

We now motivate $L_k(p_n; n) \sim \theta_k n^{-k}$, with proof provided shortly after Lemma \ref{intL1}. Consider that $\epsilon_0(x)$ converges to zero monotonically, i.e., $\epsilon_0^{\prime}(x)<0$ for $x>5.022$. Thus, for $x>3.209$, Eq. \eqref{int5} can be re-written as the following using Eq. \eqref{int1}:
\begin{equation} \label{int3}
\frac{\theta_k\alpha(x, k; p_n) }{[(1+\varphi(x))\pi(x)]^k}\lambda_1(x)<L_k(p_n; n, x)<\frac{\theta_k\alpha(x, k; p_n) }{[(1-\varphi(x))\pi(x)]^k}\lambda_2(x),
\end{equation}
where
\begin{equation*}
\begin{split}
&\lambda_1(x)=\frac{1-\frac{1}{\log^2(x)}}{1+\frac{1}{2\log^2 \left(0.8576\alpha(x, 1; p_n), \log x\right)}}, \\
&\lambda_2(x)=\frac{1+\frac{1}{2\log^2(x)}}{1-\frac{1}{\log^2\left(0.8576\alpha(x, 1; p_n), \log x\right)}},
\end{split}
\end{equation*}
and $(1+\max \epsilon_0(p_n))^{-1}=0.8576$. The following Lemma dilutes the bounds in Eq. \eqref{int3} for all $x>x_i, i=1,2,3,4.$
\begin{lemma} \label{intL1} We have the following:
\begin{enumerate}
\item there exists constants $A_1, A_2, \tilde{A_2}>0$ such that 
\begin{enumerate} [label=(\roman{*})]
\item if $\alpha(x, 1; p_n)\sim 1$, then $\forall x \geqslant  x_1>0$ we find 
\begin{equation*}
\lambda_1(x)>1-\frac{A_1}{\log^2(x)},
\end{equation*}
\item if $\alpha(x, 1; p_n)\nsim 1$, then $\forall x \geqslant  x_2>0$ we find
\begin{equation*}
\lambda_1(x)>\tilde{A_2}-\frac{A_2}{\log^2(x)};
\end{equation*}
\end{enumerate}
\item there exists constants $A_3, A_4, \tilde{A_4}>0$ such that 
\begin{enumerate} [label=(\roman{*})]
\item if $\alpha(x, 1; p_n)\sim 1$, then $\forall x \geqslant  x_3>0$ we find
\begin{equation*}
\lambda_2(x)<1+\frac{A_3}{\log^2( \log x)},
\end{equation*}
\item if $\alpha(x, 1; p_n)\nsim 1$, then $\forall x \geqslant  x_4>0$ we find
\begin{equation*}
\lambda_2(x)<\tilde{A_4}+\frac{A_4}{\log^2(x)}.
\end{equation*}
\end{enumerate}
\end{enumerate}
\end{lemma}
\begin{proof} $\newline $(1)(i). If we take, for instance, $x_1=5.022$, then some $A_1$ can be chosen where
\begin{equation*}
\begin{split}
\log^2(x)(1-\lambda_1(x)) &\geqslant  0.9410 \log^2(x)+0.0590 \\
&>A_1>0
\end{split}
\end{equation*}
because when $K$ is sufficiently small, 
$$\frac{1}{1+\frac{1}{2\log^2\left(0.8576\alpha(x, 1; p_n) \log x\right)}}>0.0590,\hspace{0.5cm} \mbox{$x>5.022.$}$$
(1)(ii). If $\alpha(x, 1; p_n) \nsim 1$, then $p_n$ can be fixed. Thus, any constant
$$\tilde{A_2}, A_2< \frac{1}{1+\frac{1}{2\log^2 \left(0.8576\log p_n\# \right)}}$$
can be chosen for some $3.209<p_n\# \leqslant  (x-g)/K$.$\newline $
(2)(i). By the series $\sum_{j=0}^{\infty}x^j(1-x)^{-1}$, we have
\begin{equation*}
\begin{split}
\lambda_2(x)&=\left(1+\frac{1}{2\log^2(x)}\right) \sum_{j\geqslant  0} \log^{-2j}\left(0.8576\alpha(x, 1; p_n) \log x\right)\\
&=1+O\left( \log^{-2}\left(0.8576\alpha(x, 1; p_n) \log x\right)\right) \\
&<1+\frac{A_2}{\log^2( \log x)}, \hspace{0.3cm} A_2>0.
\end{split}
\end{equation*}
(2)(ii). If $\alpha(x, 1; p_n) \nsim 1$, then $p_n$ can be fixed. Thus, any constant
$$\tilde{A_4}, A_4> \frac{1}{1-\frac{1}{\log^2 \left(0.8576\log p_n\# \right)}}$$
can be chosen for some $3.209<p_n\# \leqslant  (x-g)/K$.
\end{proof}
Theorem \ref{U} is proved now. 
\begin{proof}
By the above Lemma, 
\begin{equation} \label{int4}
\begin{split}
\frac{\theta_k\alpha(x, k; p_n) }{[(1+\varphi(x))\pi(x)]^k}&\left(1-\frac{A_1}{\log^2(x)}\right)<L_k(p_n; n) \\
&<\frac{\theta_k\alpha(x, k; p_n) }{[(1-\varphi(x))\pi(x)]^k}\left(1+\frac{A_2}{\log^2( \log x)}\right), \hspace{0.3cm} \mbox{$x>x_1, x_2.$}
\end{split}
\end{equation}
Without proof, $\lim_{x \to \infty} (1\pm \varphi(x))^{-k} =1$, and from the definition of $\varphi(x)$ and Lemma \ref{intL1}, we find
\begin{equation*}
\begin{split}
&\lim_{x \to \infty} \left[\left(\frac{1}{1+\varphi(x)}\right)^k\left(1-\frac{A_1}{\log^2(x)}\right)\right]=1, \\
&\lim_{x \to \infty} \left[\left(\frac{1}{1-\varphi(x)}\right)^k\left(1+\frac{A_2}{\log^2( \log x)}\right)\right] =1.
\end{split}
\end{equation*}
Hence, for every $\varepsilon \in (0, 1)$, there exists a $\tilde{x}>x_1, x_2$ such that 
\begin{equation*}
\left.
\begin{split}
1-\varepsilon<\left(\frac{1}{1+\varphi(x)}\right)^k\left(1-\frac{A_1}{\log^2(x)}\right)<1+\varepsilon& \\
1-\varepsilon<\left(\frac{1}{1-\varphi(x)}\right)^k\left(1+\frac{A_2}{\log^2( \log x)}\right)<1+\varepsilon&
\end{split}
\right\} \forall x>\tilde{x}.
\end{equation*} 
Therefore,
$$1-\varepsilon<\frac{L_k(p_n; n)}{\theta_k\alpha(x, k; p_n)/\pi(x)^k}<1+\varepsilon, \hspace{0.3cm} \mbox{$\forall x>\tilde{x}$,}$$
which implies that
$$\lim_{x \to \infty}\frac{L_k(p_n; n)}{\theta_k\alpha(x, k; p_n)/\pi(x)^k}=1.$$
It is is easy to find $\theta_k$. Straightforwardly, if $k=1$, then $\theta_1=c^{-1}$. Albeit, more complications occur when $k=2$. To start, 
\begin{equation*}
\prod_{\substack{p \leqslant  x^c \\ p > 2}} \frac{p-2}{p}=\prod\limits_{\substack{p \leqslant  x^c\\ p>2}}\frac{p(p-2)}{(p-1)^2}\prod_{\substack{p \leqslant  x^c \\ p > 2}} \bigg(\frac{p-1}{p}\bigg)^2.
\end{equation*}
Therefore, 
\begin{equation*} \label{M}
\theta_2=\frac{1}{c^2} \prod\limits_{\substack{p \leqslant  x^c \\ p>2}}\frac{p(p-2)}{(p-1)^2} \bigg(\prod\limits_{\substack{p \leqslant  p_n \\ p>2}}\frac{p(p-2)}{(p-1)^2}\bigg)^{-1}.
\end{equation*}
The notation set-forth in 1923 by Hardy and Littlewood \cite{8} describes
\begin{equation*}\prod_{\substack{p \leqslant  \sqrt{x} \\ p>2}}\frac{p(p-2)}{(p-1)^2} =\Pi_2, \hspace{0.5cm} \mbox{as }  x \to \infty,
\end{equation*}
where $\Pi_2\approx0.66016$ is known as the twin prime constant; Wrench \cite{15} truncated this constant at $42$ digits. Accordingly,  $ \theta_2 \to c^{-2}$ when both $x, p_n  \gg 1$. In any case, we have $(15/16)<\theta_2\leqslant  c^{-2}$ when considering all primorials $p_n\#$, for $n>1$.

A lot of literature uses $c=1/2$ so that Eq. \eqref{mer} is a maximum for a given $x$. But, there are other common values for $c$. For instance, P\'{o}lya \cite{11} suggests that $c=e^{-\gamma}$ to maintain consistency with the PNT, and we can always take $c=1$. Nevertheless, for the sake of generality, we allow $c$ to be chosen from the interval $[1, 1/2]$, which includes $e^{-\gamma}$. We say $(15/16)^{k-1}<\theta_k \leqslant  2^{k}.$
\end{proof}
Corollary \ref{cor1} follows without proof from Theorem \ref{U} and Lemma \ref{intL1}.
\begin{corollary} \label{cor1}
Consider $x=Kp_n\#+g$ and suppose $\alpha(x, k ;p_n) \nsim 1$. We have:
\begin{enumerate}
\item $\tilde{A_3}\theta_k\alpha(x, k ;p_n)n^{-k} \leqslant  L_k(p_n; n, x) \leqslant  \tilde{A_4}\theta_k\alpha(x, k ;p_n)n^{-k};$
\item $ L_k(p_n; n, x) \to \theta_k\alpha(x, k ;p_n)n^{-k}, \hspace{0.5cm} \mbox{$p_n\# \to \infty, p_n\# \leqslant  (x-g)/K$}$.
\end{enumerate}
\end{corollary}

\subsection{Proof of Theorem \ref{Aa}}
Several bounds for Mertens' Theorem are provided now. In 2016, Dusart \cite{16} sharpened the unconditional bounds of Rosser and Schoenfeld (see \cite[Theorem 8, p.g. 73]{13}) to
\begin{equation} \label{dus}
\frac{e^{-\gamma}}{\log x}\left(1-\frac{1}{5(\log x)^3}\right)<M^*(x)<\frac{e^{-\gamma}}{\log x}\left(1+\frac{1}{5(\log x)^3}\right)
\end{equation}
for $x\geqslant  2\ 278\ 382$. Tighter bounds are currently given by Axler \cite{17}:
\begin{equation}\label{axl}
M^*(x)>\frac{e^{-\gamma}}{\log x}\left(1-\frac{1}{20(\log x)^3}-\frac{3}{16(\log x)^4}\right),
\end{equation}
which holds for $x>46\ 909\ 038,$ and
\begin{equation}\label{axl2}
M^*(x)<\frac{e^{-\gamma}}{\log x}\left(1+\frac{1}{20(\log x)^3}+\frac{3}{16(\log x)^4}+\frac{1.02}{(x-1)\log x}\right), 
\end{equation}
for $x>1$. Proof of Theorem \ref{Aa} follows.
\begin{proof}

Let $x=p_n\#\pm1$, and take $p_n\#>8\cdot10^{989\ 079}$ because $(p_{168\ 064}=2\ 278\ 379)\#<8\cdot10^{989\ 079}<(p_{168\ 065}=2\ 278\ 421)\#$. Simply, 
\begin{equation*}
\begin{split}
1-4.233\cdot10^{-21}<\frac{\log x}{e^{-\gamma}}M^*&(x)<1+4.233\cdot10^{-21},\\
1-6.375\cdot10^{-5}<\frac{\log p_n}{e^{-\gamma}}M^*&(p_n)<1+6.375\cdot10^{-5},
\end{split}
\end{equation*}
where the top lower bound is found by Eq. \eqref{axl}, the bottom lower bound by Eq. \eqref{dus}, and both upper bounds by Eq. \eqref{axl2}. We take
$\alpha(x, k; p_n) =1$, and 
$$
(1-6.42\cdot10^{-5})^k \frac{\theta_k}{n^k}< L_k(p_n; n)<(1+6.42\cdot10^{-5})^k \frac{\theta_k}{n^k}
$$
follows from Eq. \eqref{int1} because
\begin{equation*}\label{s5}
\begin{split}
\varphi(8\cdot 10^{989\ 079})&=\left|(1-4.39\cdot10^{-7})\left(1-\frac{9.433}{e^{593.984}}\right)-1\right|\\
&\approx 4.39\cdot10^{-7}.
\end{split}
\end{equation*}
 Theorem \ref{Aa} follows easily. 
\end{proof}

\subsection{Proof of Theorem \ref{denns}}
\begin{proof}
Let 
$$
\mathcal{\overline{P}}=\{2\} \cup \{\mbox{some, but not all, primes $p \leqslant  \sqrt{\overline{x}}$}\}.$$
For $K_1 \in \mathbb{N^+}$, we can always write
\begin{equation}
\overline{x}=\left(K_1\prod_{p \in \mathcal{\overline{P}}}p\right) \pm1.
\end{equation}
Let $|\mathcal{P}_{x}|=|\mathcal{P}_{\overline{x}}|=d$. Then
\begin{equation} \label{pol2}
\left( \prod\limits_{p \leqslant  p_n} \frac{p}{p-k} -\prod\limits_{p \in \mathcal{\overline{P}}}\frac{p}{p-k}\right)\prod_{p \leqslant  p_d} \frac{p-k}{p} >0
\end{equation}
for all $x$ and qualifying $\overline{x}$ if and only if Theorem \ref{denns} holds. 
Equation \eqref{pol2} is easily simplified because $\mathcal{P}_{x}$ and $\mathcal{P}_{\overline{x}}$ only differ by $(\mathcal{P}_{x} \setminus \mathcal{\overline{P}}) \cup (\mathcal{\overline{P}} \setminus \mathcal{P}_{x})$. Define the sets
\begin{equation*}
\begin{split}
&\mathcal{P}_{x} \setminus \mathcal{\overline{P}}=\{p_a, p_b,...,p_c\}, \\
&\mathcal{\overline{P}} \setminus \mathcal{P}_{x}=\{p_d, p_e,...,p_f\}, \\
&\mathcal{P}_{x} \cap \mathcal{\overline{P}}=\{p_g, p_h,...,p_i\}.
\end{split}
\end{equation*}
We find that
\begin{equation*}\label{bs2}
\prod\limits_{j=a,b,...,c} \frac{p_j}{p_j-k}-\prod\limits_{s=d,e,...,f} \frac{p_s}{p_s-k}
\end{equation*}
will always have the same sign as Eq. \eqref{pol2}. We have 
\begin{equation*} \label{bs3}
 \prod_{l=g,h,...,i} \prod_{j=a,b,...,c} p_jp_l=TK_1 \prod_{l=g,h,...,i} \prod_{s=d,e,...,f} p_sp_l.
\end{equation*}
because $x=T\overline{x}$, for $1/4< T < 4$, as a condition for $\mathcal{P}_{x}=\mathcal{P}_{\overline{x}}$ and the term $\pm1$ in $\overline{x}$ can be eliminated by simply choosing the term $\pm1$ in $x$ to have the same sign; the latter simplification being justified because $L_k(p_n\#-1)=L_k(p_n\#+1)$ by definition.
Hence,
\begin{equation*} \label{bs4}
\prod\limits_{j=a,b,...,c} \frac{1}{p_j-k}-\frac{1}{TK_1}\prod\limits_{s=d,e,...,f} \frac{1}{p_s-k}
\end{equation*}
will also have the same sign as Eq. \eqref{pol2}. Take
\begin{equation*}TK_1= \prod_{y=1}^{|\mathcal{\overline{P}} \setminus \mathcal{P}_{x}|} z_y,
\end{equation*}
and it goes that all $z_y$ are positive constants such that the product of a particular combination of $p_j$ equals the product of only some particular $p_s$ and some particular $z_y$. No matter how we choose to arrange this, there is always exactly one $z_y$ for every $p_s$ because $z_y$ is arbitrary. There are two cases to consider: $|\mathcal{\overline{P}}|=n$, call this Case (1), and $|\mathcal{\overline{P}}|<n$, call this Case (2). If $|\mathcal{\overline{P}}|>n$, then $\mathcal{P}_{x} \neq \mathcal{P}_{\overline{x}}$ (particularly $|\mathcal{P}_{x}| < |\mathcal{P}_{\overline{x}}|$), which contradicts $\mathcal{P}_{x} = \mathcal{P}_{\overline{x}}$. Each case is considered now.

Case (1). If $|\mathcal{\overline{P}}|=n$, then there is exactly one $p_j$ for every $p_s$ and $z_y$ pair. Any $p_j$ and $p_s$ can be found such that $p_j\leqslant  p_s$ and $p_j=z_yp_s$, implying that
\begin{equation} \label{bs1}
(p_j-k)\leqslant  z_y(p_s-k)
\end{equation}
because $z_y \leqslant  1, \forall y.$ It must also be true that $z_y <1$ for at least one $y$ or else $\overline{x}$ contradicts its definition.
Because $j$, $s$, and $y$ are arbitrary, Eq. \eqref{bs1} holds for all such values. Therefore, Eq. \eqref{pol2} maintains the sign $>0$ always for the case that $|\mathcal{\overline{P}}|=n$.

Case (2). If $|\mathcal{\overline{P}}|<n$, then there is exactly one $p_j$ for exactly $|\mathcal{\overline{P}}|-1$ pairs of $p_s$ and $z_y$. For these $|\mathcal{\overline{P}}|-1$ pairs, any $p_j$ and $p_s$ can be found such that $p_j\leqslant  p_s$ and $p_j=z_yp_s$. Case (1) proves that Eq. \eqref{bs1} holds for these such values. That leaves one remaining $p_s$ and $z_y$ value each to consider, which we denote by $p_s^{\prime}$ and $z_y^{\prime}$, and $n-|\mathcal{\overline{P}}|+1$ values of $p_j$ left to consider. Let $C_0, C_1, ..., C_{n-|\mathcal{\overline{P}}|}$ denote the $n-|\mathcal{\overline{P}}|+1$ values of $p_j$, and let $U, V, W \in \mathbb{N^+}$ for $1 \leqslant  U,V, W \leqslant  n-|\mathcal{\overline{P}}|$ and $U < V$. In general, 
$$\prod_{i=0}^U (C_i-k) > \prod_{i=0}^V (C_i-k).$$
Thus, if
$$\prod_{i=0}^1 (C_i-k) < z_y^{\prime}(p_s^{\prime}-k),$$ then $$\prod_{i=0}^W (C_i-k) < z_y^{\prime}(p_s^{\prime}-k).$$ Since $C_0C_1=z_y^{\prime}p_s^{\prime}$, we have
\begin{equation*}
\begin{split}
\prod_{i=0}^1 (C_i-k)-z_y^{\prime}(p_s^{\prime}-k) &= (C_0-k)(C_1-k)-z_y^{\prime}(p_s^{\prime}-k) \\
&=-C_0-C_1+k+z_y^{\prime}.
\end{split}
\end{equation*}
 Multiplying through by $p_s^{\prime}$ yields $-C_0p_s^{\prime}-C_1p_s^{\prime}+k p_s^{\prime}+C_0C_1=-C_0p_s^{\prime}+k p_s^{\prime}+C_1(C_0-p_s^{\prime}).$ But, $C_0<p_s^{\prime}$ and $C_0>k$, so $-C_0p_s^{\prime}+k p_s^{\prime}+C_1(C_0-p_s^{\prime})<0$. Thus, $(C_0-k)(C_1-k)<z_y^{\prime}(p_s^{\prime}-k)$, implying that Eq. \eqref{pol2} maintains the sign $>0$ always for the case that $|\mathcal{\overline{P}}|<n$. Therefore, Theorem \ref{denns} is proved.
\end{proof}

\section{Conjectures on Primes and Twin Prime Pairs}
Caldwell and Gallot \cite{3} note that the probability that some number $p_n\# \pm 1$ is prime can be estimated heuristically by dividing $(\log (p_n\# \pm 1))^{-1} \approxeq 1/p_n$ by $M^*(p_n) \sim e^{-\gamma}/ \log p_n$ to yield $e^{\gamma}\log p_n/ p_n$. Then, the expected number of primorial primes of each of the forms $p_n\# \pm 1$ less than or equal to $p_N\#\pm1$, respectively, is given by
\begin{equation} \label{cad}
\sum_{p\leqslant  p_N} \frac{e^{\gamma}\log p}{ p} \sim e^{\gamma}\log p_N.
\end{equation}
(Multiplying this result by two yields the total expected number of primorial primes less than or equal to $p_N\#+1$.) This result motivates their conjecture below.
\begin{conjecture} [\bfseries Caldwell and Gallot]\label{cn2}
The expected number of primorial primes of each of the forms $p_n\# \pm 1$ less than or equal to $p_N\# \pm 1$, respectively, are both approximately $e^{\gamma}\log p_N$. 
\end{conjecture}

Theorem \ref{U} shows that their estimate $e^{\gamma}\log p_n/ p_n$, and thus Eq. \eqref{cad}, is a $\theta_1e^{-\gamma}$ term off when $p_n\gg1;$ the only exception being that their estimate obtains equality with Eq. \eqref{asim} if $c=e^{-\gamma}$ (P\'{o}lya's bound). Therefore, we suggest that $e^{\gamma}\log p_N$ in Conjecture \eqref{cn2} be revised to be $\theta_1\log p_N$ where $\theta_1$ is a constant in the interval $[1, 2]$, as defined previously. Table 1 provides the actual and expected number of primorial primes less than or equal to $p_N\#+1$. Note that the third column in Table 1 produces an interval of values that is consistent with the possible values of $\theta_1$.
\begin{table}[hbt!]
\centering
\caption{Actual and Expected number of primorial primes less than or equal to $p_N\#+1$.}
 \begin{tabular}{|c|c|c|c|} 
 \hline
 $N$ & Actual & Expected $(2\theta_1\log p_N)$ & Expected $(2e^{\gamma}\log p_N)$\\ [0.5ex] 
 \hline
$10$ & $9$ &$[6.74, 13.47]$ & $12$\\ 
$100$ &  $15$  &$[12.58, 25.17]$ & $22.40$\\
$1,000$ &  $29$  &$[17.96, 35.90]$ & $31.98$\\
$10,000$ &  $37$ &$[23.12, 46.24]$ & $41.12$\\
$100,000$ &  $\geqslant 42$ &$[28.16, 56.31]$ & $50.14$ \\ [1ex] 
 \hline
\end{tabular}
\end{table}
Although this data is rather limited, it is successful in supporting Conjecture \ref{cn2} and maintaining our suggested revision. That being said, the largest value in the interval is currently the most accepted because it is consistent with taking $c=1/2$.

Similarly, the expected number of primorial twin prime pairs less than or equal to the pair $(p_N\#-1, p_N\#+1)$ is 
\begin{equation} \label{ncad}
\sum_{p\leqslant  p_N} \theta_2\bigg(\frac{\log p}{ p} \bigg)^2.
\end{equation}
Immediately, we notice that Eq. \eqref{ncad} converges to a constant, say $\theta_2\Omega$, as $N \to \infty$. 

Elementary bounds for $\Omega$ are found now. For $x \geqslant  17$, the following simple inequality holds \cite{13}: 
\begin{equation*}
0.796775\frac{\log x}{ x} < \frac{1}{\pi(x) }< \frac{\log x}{x}.
\end{equation*} We find
$$ \sum_{p\leqslant  (p_{17}=59)} \bigg(\frac{\log p}{ p} \bigg)^2 \approx 0.660163 $$
and
$$\sum_{n > 17} \frac{1}{n^2}\approx 0.057134.$$
Hence, an elementary bounds for $\Omega$ follows:
$$
0.717297<\Omega<0.750159.
$$
Let $\theta_2=4$. Then, the total expected number of primorial twin prime pairs is in the interval $(2.8692, 3.0064)$. By taking this number to be three--the only integer in the interval--Conjecture \ref{cn1} is initiated, with more justification to come. Computer programs are used to estimate the convergences of $\Omega$, and Table 2 shows various values of $\Omega$ for a given $N$.

\begin{table}[hbt!]
\centering
\caption{Evaluation of $\Omega$ for a given $N$.}
 \begin{tabular}{|c|c|} 
 \hline
 $N$ & Value of $\Omega$ \\ [0.5ex] 
 \hline
$10$ & $0.605414$ \\ 
$100$ & $0.728261$\\
$1,000$ & $0.740344$\\
$10,000$ & $0.741478$\\
$100,000$ & $0.741586$ \\ [1ex] 
 \hline
\end{tabular}
\end{table}

\begin{table}[hbt!]
\centering
\caption{Actual and Expected number of primorial twin prime pairs less than or equal to the pair $(p_N\#-1, p_N\#+1)$.}
 \begin{tabular}{|c|c|c|} 
 \hline
 $N$ & Actual & Expected  \\ [0.5ex] 
 \hline
$10$ &  $3$ &$2.42$ \\ 
$100$ &  $3$ &$2.91$ \\
$1,000$ & $3$ &$2.96$ \\
$10,000$ & $3$ &$2.97$ \\
$100,000$ &  $3$ &$2.97$ \\ [1ex] 
 \hline
\end{tabular}
\end{table}

The various values of $\Omega$ in Table 2 are used to find the expected number of primorial twin prime pairs less than or equal to the pair $(p_N\#-1, p_N\#+1)$. This data is provided in Table 3 along with the actual number. We find that this data is successful in supporting Conjecture \ref{cn1}. The three primorial twin prime pairs are given by $n=2,3,5$.

Similar results are obtained for universal primorials. Let $x$ be a large universal primorial. We have 
$$L_2(p_n;n)=\theta_2\left(\frac{\log p_n}{ \log x }\right)^2\approxeq \theta_2\left(\frac{\log p_n}{\vartheta(p_n)+\log (N/2)}\right)^2.$$ Since $N>(2+\log (N/2))^2$ for $N > 18$ (the exact value is $17.262$), we find that the expected number of instances where $x$ and $x+2$ are both prime is 
\begin{equation*}\label{split}
\begin{split}
\approxeq \theta_2\sum_{p}\sum_{N=1}^{\infty}\bigg(\frac{\log p}{p+\log (N/2)}\bigg)^2 &> \sum_{N=1}^{\infty} \bigg(\frac{\log 2}{2+\log (N/2)}\bigg)^2\\
& >1.0659+\sum_{N=19}^{\infty} \frac{\log^2(2)}{N}, \\
 \end{split}
\end{equation*}
which diverges by the harmonic series. Therefore, Conjecture \ref{up} is justified.

\section{Proof of Brun's Theorem}
Brun \cite{18} provides a natural upper bound on the number of twin primes less than or at a given number $x$, which is commonly referred to as Brun's Theorem. 
\begin{theorem} [\bfseries Brun]For $x \geqslant  3$,
\begin{equation} \label{brun}
\pi_2(x)=O\bigg( \frac{x(\log\log x)^2}{\log^2(x)}\bigg).
\end{equation}
\end{theorem}
Lemma \ref{eq} and Theorem \ref{denns} together supply an alternative proof to Brun's Theorem, the former being provided and proved quickly below. 
\begin{lemma} \label{eq}
Consider $x=Kp_n\#+g$ and suppose $\alpha(x, k; p_n) \sim 1$. We have
\begin{equation*}\label{gllg3}
\frac{\log p_n}{\log x}=\frac{\log\log x}{\log x}(1+o(1)).
\end{equation*}
\end{lemma} 
\begin{proof}
Since $\alpha(x, k; p_n)\approxeq 1$ for large enough $x$, we have$$
\frac{\log\log x-\log(1+\epsilon_0(p_n))}{\log x}<\frac{\log p_n}{\log x} <\frac{\log\log x-\log(1-\epsilon_0(p_n))}{\log x}, $$
for $p_n \geqslant  149$. Take the bounds given by Dusart in Eq. \eqref{d1} so that $\epsilon_0(599)=0.14271$ and $\varphi(599)=0.30543$, yielding
$$\frac{\log\log x}{\log x}\left(1-\frac{0.1334}{\log \log x}\right)<\frac{\log p_n}{\log x} 
<\frac{\log\log x}{\log x}\left(1+\frac{0.15398}{\log \log x}\right).$$
Lemma \ref{eq} follows immediately. 
\end{proof}
The derivation of Brun's Theorem is shown now. Let $k=2$ and suppose that $\alpha(x, k; p_n) \sim 1$. By Lemma \ref{eq} and Theorem \ref{denns}, $$\pi_2(x)  \ll B_0\sum_{j=2}^{x}\left(\frac{\log\log j}{\log j}\right)^2, $$ where $B_0>0$ is a constant. Thus,
\begin{equation*}
\begin{split}
\pi_2(x) & \ll B_0\left(\frac{\log\log x}{\log x}\right)^2\sum_{j=2}^{x} \left(\frac{\log x(\log\log j)}{(\log\log x)\log j}\right)^2= O\bigg( \frac{x(\log\log x)^2}{\log^2(x)}\bigg).
\end{split}
\end{equation*}

\section{Auxiliary Results}
The below results follow without proof from the results obtained in Sections 3.1--3.3.
\begin{lemma} \label{coo2} For $x \geqslant  3\cdot 10^{120}$, 
\begin{equation*}
\left| \frac{1}{\pi(x)}-\frac{\log x}{\vartheta(x)} \right | < 0.0050222\frac{\log x}{\vartheta(x)}.
\end{equation*}
\end{lemma}
\begin{lemma} \label{coo3} For $x \geqslant  8\cdot 10^{989\ 079}$, 
\begin{equation*}
\left| \frac{1}{\pi(x)}-\frac{\log x}{\vartheta(x)} \right | < 4.39\cdot10^{-7}\frac{\log x}{\vartheta(x)}.
\end{equation*}
\end{lemma}
\begin{corollary}\label{plo1}
For any primorial $p_n\#$, $n>5$,
\begin{equation*}\label{uio12}
\theta_k\left(\frac{0.78482}{n}\right)^k< L_k(p_n; n)<\theta_k\left(\frac{1.46135}{n}\right)^k.
\end{equation*}
\end{corollary}
\begin{corollary}\label{plo2}
For any primorial $p_n\#$, $n>62$,
\begin{equation*}\label{uio14}
\theta_k\left(\frac{ 0.99392}{n}\right)^k< L_k(p_n; n)<\theta_k\left(\frac{1.02089}{n}\right)^k.
\end{equation*}
\end{corollary}

\end{document}